\documentclass[11pt]{amsart} \textwidth=14.5cm \oddsidemargin=1cm
\usepackage{amsmath,wasysym}
\usepackage{amsxtra}
\usepackage{amscd}
\usepackage{amsthm}
\usepackage{amsfonts}
\usepackage{amssymb}
\usepackage{eucal}
\usepackage{graphics, color}
\usepackage{hyperref,mathrsfs}
\usepackage[usenames,dvipsnames]{xcolor}
\usepackage{tikz}
\usepackage{tikz-cd}
\usepackage{stmaryrd}
\usepackage[all]{xy}
\usepackage{hyperref}
\usepackage{color}
\usepackage{bbm} 
\usepackage{upgreek}
\usepackage{float}
\usepackage{dynkin-diagrams}
\usepackage{rotating}
\allowdisplaybreaks
\usepackage{cite}
\hypersetup{colorlinks,linkcolor=blue, urlcolor=blue,
	citecolor=blue}

\usepackage[all]{xy}

\newcommand{\BF}{{\mathbb {F}}}

\newcommand{\BN}{{\mathbb {N}}}
\newcommand{\BO}{{\mathbb {O}}}

\newcommand{\BZ}{{\mathbb {Z}}}

\newcommand{\calD}{{\mathfrak {D}}}

\newcommand{\calO}{{\mathfrak {O}}}

\newcommand{\caln}{{\mathfrak {n}}}

\newcommand{\RU}{{\mathrm {U}}}

\newcommand{\ScP}{{\mathscr {P}}}

\newcommand{\Ind}{{\textup{Ind}}}

\newcommand{\Lie}{{\mathrm{Lie}}}
\newcommand{\Mat}{{\mathrm{Mat}}}

\newcommand{\rank}{{\mathrm{rank}}}

\newcommand{\Supp}{{\mathrm{Supp}}}

\newcommand{\RMat}{{\mathrm{Mat}}}
\newcommand{\tr}{{\mathrm{tr}}}

\newtheorem*{thm*}{Theorem}
\newtheorem{thm}{Theorem}[section]

\newtheorem{lem}[thm]{Lemma}

\newtheorem{conj}[thm]{Conjecture}

\newtheorem{ques/conj}[thm]{Question/Conjecture}

\newtheorem{defn}[thm]{Definition}
\newtheorem{rmk}[thm]{Remark}

\newcommand{\A}{\mathbb{A}}

\input xy
\xyoption{all}

\usepackage{hyperref}

\numberwithin{equation}{section}

\def\calD{\mathcal{D}}

\def\calO{\mathcal{O}}

\newcommand\frP{\mathfrak{P}}

\renewcommand\a\alpha
\renewcommand\b\beta
\newcommand\g\gamma
\renewcommand\d\delta

\title[polynomial properties of unipotent radicals]{Polynomial properties of unipotent radicals of parabolic subgroups in classical groups}

\author{Qingchun Hao}
\author{Yang Yang}

\address{School of Mathematical Sciences, East China Normal University, Shanghai, 200241, P. R. China}
\email{52275500011@stu.ecnu.edu.cn (Yang)}

\date{}

\begin{document}
  
\maketitle

	 \begin{abstract}
            Let $R_u^{X,d}$ denote the unipotent radical of a (proper) maximal standard parabolic subgroup of the classical group $\mathrm{Sp}_{2n}(q)$, $\mathrm{SO}_{2n}(q)$, or $\mathrm U_{2n}({q^2})$. This paper establishes explicit formulas for the number of irreducible characters of $R_u^{X,d}$ with degree $q^e$   .
            
	 \end{abstract}   

	  \maketitle
	
	 
	 \section{Introduction}
	Let $\BF_q$ be the finite field of $q$ elements, where $q$ is a power of an odd prime. Denote by $U_n(q)$ the unitriangular group of degree $n$ over $\BF_q$. Let $\calD$ be a subset of $\Delta_n:= \{(i,j) \mid 1\le i < j \le n\}$ which is closed in the sense that if $(i,j),(j,k) \in \calD $ then $(i,k) \in \calD$. A pattern group $G_{\calD}(q)$ is a subgroup of $U_n(q)$ consisting of the matrices whose $(i,j)$-th entry is nonzero only if $(i,j) \in \calD$. The classification of all irreducible representations of all pattern groups is
 a wild problem. Even counting the number of irreducible characters of such groups is not easy. By a result due to Isaacs \cite{Isa95},  it is known that the dimension of a complex irreducible representation of a pattern
group is always a power of $q$. In \cite{Leh74}, Lehrer gave a conjecture as follows:
	 \begin{conj}[\bf Lehrer's conjecture]
	 	The number of irreducible characters of $U_n(q)$ with degree $q^{e}$ $(e\in \BN)$ is a polynomial in $q$ with integral coefficients.
	 \end{conj}

 Let $N_{n,e}(q)$  denote the number of irreducible characters of $U_n(q)$
with degree $q^e$. Isaacs gave a strengthened form of Lehrer's conjecture in \cite{Isa07}. 
 \begin{conj}[\bf Isaacs's conjecture]\label{IC}
The functions $N_{n,e}(q)$ are polynomials in $q-1$ with non-negative integral coefficients.
	 \end{conj}

 The Sylow $p$-subgroups of the general linear groups are the groups of unitriangular matrices. There is an extensive literature on these groups and, specifically, on their characters. Isaacs showed that the degrees of the irreducible characters are always powers of $q$, the size of the underlying field \cite[Corollary B]{Isa95}. In paper \cite{Isa07} this author has investigated the irreducible characters of maximal and second maximal degree and computed the exact number of such characters. In \cite{Sa09}, Sangroniz delineated the Sylow $p$-subgroups of classical groups and their corresponding Lie algebras, further exploring the irreducible characters of the highest degree Sylow $p$-subgroup within these classical groups. Inspired by his work, we study the irreducible characters of unipotent radicals in maximal standard parabolic subgroups of classical groups. We refer to a group as admitting the polynomial property if an analogue of Lehrer’s conjecture holds for it. That is, the number of irreducible characters of a group with degree $q^{e}$ is a polynomial in $q$ with integral coefficients. In \cite{Ni22r}, Nien proved the polynomial property for unipotent radicals of standard parabolic subgroups with 3 parts.
The motivation of the paper is an extension of the works of Lehrer-Isaacs-Nien.

Let $G$ be the symplectic group $\mathrm {Sp}_{2n}(q)$ (type $\mathrm C$) or even special orthogonal group $\mathrm {SO}_{2n}(q)$ (type $\mathrm D$) defined over $\BF_q$, or even unitray group $\mathrm U_{2n}(q^2)$ (type $\mathrm U$) defined over $\BF_{q^2}$ (the notation for the even unitary group is typically $^{2}A_{2n}$, but for the sake of convenience in writing, we shall denote it by $\mathrm U$ in this context). Denote by ${R_u^{X,d}}$ the unipotent radical of a maximal standard parabolic subgroup in the group of type $X$, where $X=\mathrm C,\mathrm D,\mathrm U$. 
Each ${R_u^{X,d}}$ has a canonical semi-direct product decomposition into two of its abelian subgroups, i.e. ${R_u^{X,d}} \simeq A^{X}\rtimes H^{X}$ (see Table \ref{2} for the abelian subgroups $H^X$ and $A^X$). We introduce a map $\pi_c$ from $\Lie( A^{X})^t$, the dual space of $\Lie(A^{X})$, to a subset of $\RMat_h(k)$ consisting of matrices of size $h$ over $k$. We give a numerical formula for the number of irreducible characters with degree $q^e$ for $e\in \BN$. In essence, this formula is a $q$-power product with $|\Tilde{\mathfrak{P}}_X({R_u}, e)|$, where $|\Tilde{\mathfrak{P}}_X({R_u}, e)|$ is the preimage of $\pi_c$ over a subset of $\RMat_h(k)$; see \ref{eqru}. The problem of the polynomiality of ${R_u^{X,d}}$ then becomes the question of whether $|\Tilde{\mathfrak{P}}_X({R_u}, e)|$ is a polynomial in $q$ with integral coefficients for $e\in \BN$. Using the methods of Clifford theory and finite field version of orbit theory, we know that $|\Tilde{\mathfrak{P}}_X({R_u}, e)|$ is a $q$-power product with the cardinality of a specific class of matrices (such as symmetric matrices for type $\mathrm C$). On the other hand, we can also use the geometric method to prove the polynomial property and get a stronger result. That is, the number of irreducible characters of ${R_u^{X,d}}$  with degree $q^e$ are polynomials in $q-1$ with non-negative integral coefficients.

This paper is organized as follows. In Section \ref{section 2}, we give the necessary preliminaries for the unipotent radicals of the maximal standard parabolic subgroups. In Section \ref{4}, we review Clifford theory and finite version orbit theory, establishing character number formulas of a given degree for all ${R_u^{X,d}}$. In Section \ref{geo poly}, we explicitly give all irreducible character number formulas for these unipotent radicals. In Section \ref{GC}, using geometric methods, we obtain a stronger result with respect to the number of irreducible characters of ${R_u^{X,d}}$  with degree $q^e$.

Throughout this paper, we use $\BN$ to represent the set of non-negative integers and $\BZ_{+}$ to denote the set of positive integers. For any finite set $S$, we express its cardinality as $|S|$. The Lie algebra of a closed linear algebraic group $A$ is represented by $\Lie(A)$. We indicate the transposition of a matrix $A$ as $A^t$ and the inverse of $A^t$ as $A^{-t}$.

	 
	 \section{Preliminaries}\label{section 2}
Let $k$ be a finite field, either $\BF_q$ or $\BF_{q^2}$, where $q$ is a power of an odd prime.
Consider the symplectic group $\mathrm{Sp}_{2n}(q)$ and the even special orthogonal group $\mathrm{SO}_{2n}(q)$ over $\BF_{q}$. Also, consider the even unitary group $\mathrm{U}_{2n}(q^2)$ over $\BF_{q^2}$. These groups are referred to as types $\mathrm C$, $\mathrm {D}$ and $\mathrm U$, respectively. \subsection{Root Systems of $\mathrm{Sp}_{2n}(q)$ and  $\mathrm{SO}_{2n}(q)$}
The Dynkin diagrams of $\mathrm{Sp}_{2n}(q)$ and  $\mathrm{SO}_{2n}(q)$ are depicted as follows:


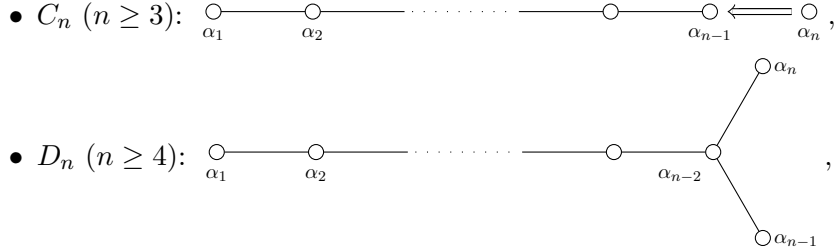
\begin{itemize}
    \item $C_n$ ($n\ge 3$):
    \dynkin[Kac,labels={\alpha_1,\alpha_2,,\alpha_{n-1},\alpha_n},
label directions={,,left,,,,right,,},
scale=2,
] C{},
\item $D_n$ ($n\ge 4$):
\dynkin[Kac,labels={\alpha_1,\alpha_2,,\alpha_{n-2},\alpha_{n-1},\alpha_n},
label directions={,,left,,,,right,,},
scale=2,
] D{},
\end{itemize}
where $\alpha_i$ are simple roots. 
     
  \subsection{Unipotent radicals of maximal standard parabolic subgroups}
  \subsubsection{The $\mathrm{Sp}_{2n}(q)$ and $\mathrm{SO}_{2n}(q)$ cases}\label{section 2.2.1}
     Let $G$ be $\mathrm{Sp}_{2n}(q)$ or $\mathrm{SO}_{2n}(q)$. Denote $T\subset G$ as the set of all diagonal elements and $B\subset G$ as the standard Borel subgroup containing $T$. We define $W$ as the Weyl group of $G$ and $S$ as the set of simple roots relative to the pair $(B,T)$. For a particular subset $S_Q\subset S$, let $Q$ be the standard parabolic subgroup associated with $S_Q$. We denote by $P^{\mathrm C,d}$ the maximal standard parabolic subgroup of $\mathrm{Sp}_{2n}(q)$ with $S\backslash\{\alpha_d\}$ as its associated set of simple roots. Furthermore, we use ${R_u^{\mathrm C,d}}$ to represent its unipotent radical and $\caln^{\mathrm C,d}$ for the nilpotent radical of the corresponding parabolic subalgebra. Similar notation ${R_u^{\mathrm D,d}}$ and $\caln^{\mathrm D,d}$ are used for $\mathrm{SO}_{2n}(q)$.
We have :
$$\caln^{X,d}:=\left\{\begin{pmatrix}
 u&v\\
 0&w \end{pmatrix}\right\}.$$
It is the semi-direct product of $\caln_{H}^{X}$ and $\caln_{A}^{X}$ ($X=\mathrm C,\mathrm D$).
The algebraic groups corresponding to $\caln_{H}^{X}$ and $\caln_{A}^{X}$ are denoted by $H^{X}$ and $A^{X}$, respectively. These will be detailed in Tables \ref{1} and \ref{2}. Consequently, ${R_u^{X,d}}$ is the semi-direct product of $H^{X}$ and $A^{X}$:
$$ {R_u^{X,d}} \simeq A^{X}\rtimes H^{X}.$$

 \subsubsection{The $\mathrm U_{2n}(q^2)$ case}\label{section 2.2.2}
 Let $G={GL}(2n,q^2)$ with Steinberg endomorphism $F=F_{q}\circ\sigma$, where $F_{q}: G\to G,(a_{ij})\mapsto(\overline{a_{ij}}):=(a^{q}_{ij})$ is the standard Frobenius map and $\sigma: G\to G,(a_{ij})\mapsto J_{2n}((a_{ij})^{-t})J_{2n}$ with
 $$J_{2n}=\begin{pmatrix}
0 & \cdots                                 & 1   \\
	\vdots &\begin{sideways}$\ddots$\end{sideways} & \vdots      \\
	1 & \cdots                                 & 0   \\
	 \end{pmatrix}.$$
The fixed points subgroup $G^F\cong\RU_{2n}(q^2)$. Denote by $W$ the Weyl group of $G$ generated by the set of simple reflections $S=\{s_i \mid 1\le i\le 2n-1\}$. Let the group $W^F$ be the Weyl group of $G^F$. This group is generated by the set $S_F$ of involutions which are in natural bijection with the $F$-orbits on $S$ (see \cite[Lemma 23.3]{Mal11}). Denote by $P_I$ the parabolic subgroup of $G$ corresponding to $I\subseteq S$ and represent $L_I$ as its standard Levi complement.  According to the double coset decomposition in \cite[Proposition 12.2]{Mal11}, $P_I$ is $F$-stable if and only if $I$ is $F$-invariant. In other words, $I$ must be a union of $F$-orbits in $S$ and correspond to a subset $I_F \subseteq S_F$. We denote $P_{I_F}$ by the parabolic subgroups of $G^F$ defined by $I_F$. In this context, we have $F(s_i)=s_{2n-i}$ ($1\le i\le 2n-1$) and $S_F=\{\beta_i=s_i s_{2n-i} , \beta_{n}=s_n \mid 1\le i\le n-1\}$, where $s_j\in S$. Below is the Satake diagram for $\RU_{2n}$:
$$\dynkin[labels*={\alpha_1,\alpha_2,,\alpha_n,,\alpha_{2n-2},\alpha_{2n-1}},
label directions={,,left,,,,right,,},
scale=3,
] A{IIIb}.$$

\begin{lem}\cite[Proposition 26.1]{Mal11}
    Let G be a connected reductive group with a Steinberg endomorphism $F$  and let $I\subseteq S$ be an $F$-invariant subset.
\begin{itemize}
    \item There is a unique $G^F$-conjugacy class of $F$-stable $G$-conjugates of $P_I$, and we have $P_I^F=P_{I_F}$.
    \item There is a unique $P_I^F$-conjugacy class of the $F$-stable Levi complements ${L_I}$ to $U_I=R_u(P_I)$ in $P_I$, and we have $P_I^{F}\simeq L_I^{F}\ltimes U_I^{F}$.
\end{itemize}
\end{lem}

The above lemma allows us to identify the maximal parabolic subgroup of $\mathrm U_{2n}(q^2)$ corresponding to $S_F\backslash\{\beta_d\}$ ($1\le d\le n-1$). We denote this group by $P^{\mathrm U,d}$ and refer to its unipotent radical as $R_u^{\mathrm U,d}$. Let $P^{d}$ be the maximal $F$-stable parabolic subgroup of ${GL}(2n,q^2)$ corresponding to $S\backslash\{\alpha_d,\alpha_{2n-d}\}$ and $R_u^{d}$ be its unipotent radical. We have 
$$
P^{d}=\left\{
\begin{pmatrix}
 A_{d}^1&A^2&A^3\\
 0&A_{2n-2d}^4&A^5\\
 0&0&A_{d}^6\\
\end{pmatrix}\Bigg|\quad \begin{aligned}
A_{d}^1,A_{d}^6 \in{GL}(d,q^2), A_{2n-2d}^4\in{GL}(2n-2d,q^2),\\ A^2\in\RMat_{d\times(2n-2d)}(q^2),
A^3\in\RMat_{d\times d}(q^2),\\A^5\in\RMat_{(2n-2d)\times d}(q^2).
\end{aligned}\right\} ,
$$
$$
R_u^{d}=\left\{
\begin{pmatrix}
 I_{d}&B^1&B^2\\
 0&I_{2n-2d}&B^3\\
 0&0&I_{d}\\
\end{pmatrix}\Bigg|\quad\begin{aligned}
    B^1\in\RMat_{d\times(2n-2d)}(q^2),\\
B^2\in\RMat_{d\times d}(q^2),\\B^3\in\RMat_{(2n-2d)\times d}(q^2).
\end{aligned}
    \right\}.
$$
Furthermore, the group $R_u^{d}$ can be decomposed as a semi-direct product in the following way: $$R_u^{d}\simeq A\rtimes H.$$
\begin{align*}
H=\left\{
\begin{pmatrix}
 I_{d}&A_1&0&0\\
 0&I_{n-d}&0&0\\
 0&0&I_{n-d}&A_2\\
 0&0&0&I_{d}
\end{pmatrix}\right\},\qquad 
    A=\left\{
\begin{pmatrix}
 I_{d}&0&B_1&B_2\\
 0&I_{n-d}&0&B_3\\
 0&0&I_{n-d}&0\\
 0&0&0&I_{d}
\end{pmatrix}\right\}.
\end{align*}

Denote by $H^{\mathrm U}$ ($A^{\mathrm U}$, respectively) the fixed points subgroup $H^F$ ($A^F$, respectively). Then, we have 
$$R_u^{\mathrm U,d}\simeq A^{\mathrm U}\rtimes H^{\mathrm{U}}.$$
We will describe them explicitly in Table \ref{2}. It follows from \cite[Proposition 23.2]{Mal11} that the semi-direct product $A^{\mathrm U}\rtimes H^{\mathrm{U}}$ is also $F$-stable.
 \subsubsection{}
Let $H^{X}=\left\{
\begin{pmatrix} 
 L&M\\
 0&N\\
 \end{pmatrix}\right\}, 
 A^{X}=\left\{\begin{pmatrix} 
 U&V\\
 0&W\\
 \end{pmatrix}\right\}$. We have: 

\begin{table}[h]
\centering
\caption{}
\begin{tabular}{|c|c|c|c|c|c|}
\hline
$X$ & $u$ &$ v$ & $w$ & $\caln_{H}^{X}$ & $\caln_{A}^{X}$  \\ \hline
$\mathrm C$ & $\begin{pmatrix}
 0&A\\
 0&0 \end{pmatrix}$ & $\begin{pmatrix}
 B_1&B_2\\
 B_3&0 \end{pmatrix}$ & $\begin{pmatrix}
 0&0\\
 -A^t&0 \end{pmatrix}$ & $\begin{pmatrix}
 u&0\\
 0&w \end{pmatrix}$ & $\begin{pmatrix}
 0&v\\
 0&0 \end{pmatrix}$  \\ \hline
$\mathrm D$ & $\begin{pmatrix}
 0&A\\
 0&0 \end{pmatrix}$ & $\begin{pmatrix}
 B_1&B_2\\
 B_3&0 \end{pmatrix}$ & $\begin{pmatrix}
 0&0\\
 -A^t&0 \end{pmatrix}$ & $\begin{pmatrix}
 u&0\\
 0&w \end{pmatrix}$ & $\begin{pmatrix}
 0&v\\
 0&0 \end{pmatrix}$  \\ \hline
\end{tabular}
\label{1}
\end{table}

\begin{table}[h]
\centering
\caption{}
\begin{tabular}{|c|c|c|c|c|c|c|}
\hline
$X$ &$ L$ & $M$ & $N$ & $U$ & $V$  & $W$ \\ \hline
$\mathrm C$ & $\begin{pmatrix}
 I_d&A\\
 0&I_{n-d} \end{pmatrix}$ & $0_{n\times n}$ & $\begin{pmatrix}
 I_d&0\\
 -A^t&I_{n-d} \end{pmatrix}$ & $I_n$ & $\begin{pmatrix}
 B_1&B_2\\
 B_3&0 \end{pmatrix}$  & $I_n$ \\ \hline
$\mathrm D$ & $\begin{pmatrix}
 I_d&A\\
 0&I_{n-d} \end{pmatrix}$ & $0_{n\times n}$ & $\begin{pmatrix}
 I_d&0\\
 -A^t&I_{n-d} \end{pmatrix}$ & $I_n$ & $\begin{pmatrix}
 B_1&B_2\\
 B_3&0 \end{pmatrix}$  & $I_n$\\ \hline
$\mathrm U$ & $\begin{pmatrix}
 I_d&A_1\\
 0&I_{n-d} \end{pmatrix}$ & $0_{n\times n}$ & $\begin{pmatrix}
 I_{n-d}&A_2\\
 0&I_d \end{pmatrix}$ & $I_n$ & $\begin{pmatrix}
 B_1&B_2\\
 0&B_3 \end{pmatrix}$  & $I_n$ \\ \hline
\end{tabular}
\label{2}
\end{table}
\begin{itemize}
    \item The matrices $A\in\Mat_{d\times (n-d)}(\BF_q)$, $A_1\in\Mat_{d\times (n-d)}(\BF_{q^2})$.
    \item The matrices $v$ and $V$ are symmetric, $B_1\in\Mat_{d\times d}(\BF_q)$, $B_2\in\Mat_{d\times(n-d)}(\BF_q)$, $B_3\in\Mat_{(n-d)\times d}(\BF_q)$, if $X=\mathrm C$.
    \item The matrices $v$ and $V$ are skew-symmetric, $B_1\in\Mat_{d\times d}(\BF_q)$, $B_2\in\Mat_{d\times(n-d)}(\BF_q)$, $B_3\in\Mat_{(n-d)\times d}(\BF_q)$, if $X=\mathrm D$.
    \item The matrices $B_1=-J_d\overline{B_3^t}J_{n-d}$, $B_2=-J_d\overline{B_2^t}J_d$, and $A_1=-J_d\overline{A_2^t}J_{n-d}$. The matrices $B_1\in\Mat_{d\times (n-d)}(\BF_{q^2})$, $B_2\in\Mat_{d\times d}(\BF_{q^2})$, $B_3\in\Mat_{(n-d)\times d}(\BF_{q^2})$, if $X=\mathrm U$.
\end{itemize}

We give some lemma which will be used in the next section.
 
\begin{lem}\label{tr}
The trace map $$\tr:\Mat_{2n}(\BF_q)\times\Mat_{2n}(\BF_q)\to \BF_q,\quad (A,B)\mapsto \tr(A\cdot B)$$ restricted to $\mathfrak{n}^\mathrm C_{{A}}\times (\mathfrak{n}^\mathrm C_{A})^t$ becomes a perfect pairing.  \end{lem}
\begin{proof}
    Let $X=\begin{pmatrix}
 0&B\\
 0&0 \end{pmatrix}\in\caln_{A}^{\mathrm C}$ and $Y=\begin{pmatrix}
 0&0\\
 A&0 \end{pmatrix}\in (\mathfrak{n}^\mathrm C_{A})^t$, where $B=\begin{pmatrix}
 B_1&B_2\\
 B_3&0 \end{pmatrix}$ and $A=\begin{pmatrix}
 A_1&A_3\\
 A_2&0 \end{pmatrix}$ are symmetric. If $\tr(X\cdot Y)=0$ for all $Y\in (\mathfrak{n}^\mathrm C_{A})^t$, we get $\tr(B_1A_1+B_2A_2+B_3A_3)=0$ for all $A$. Let $B_k=(b_{i,j}^k)$ $(k=1,2,3)$. First, taking $A_2=0, A_1=E_{i,i}\ (E_{i,j}+E_{j,i} ,\text{respectively})$, we get $b_{i,i}^1=0\ (b_{i,j}^1=b_{j,i}^1=0, \text{respectively})$. Next, taking $A_1=0$ and $A_2=E_{i,j}$, we get $b_{i,j}^2=0$. Thus, $X=0$ and the map restricted to $\mathfrak{n}^\mathrm C_{{A}}\times (\mathfrak{n}^\mathrm C_{A})^t$ is non-degenerate.
\end{proof}

\begin{lem}\label{tr2}
The trace map $$\tr:\Mat_{2n}(\BF_q)\times\Mat_{2n}(\BF_q)\to \BF_q,\quad (A,B)\mapsto \tr(A\cdot B)$$ restricted to $\mathfrak{n}^\mathrm D_{{A}}\times (\mathfrak{n}^\mathrm D_{A})^t$ becomes a perfect pairing.  \end{lem}
\begin{proof}
 The proof is similar to that of Lemma \ref{tr}.   
\end{proof}

Let $\mathfrak{n}^\mathrm U_{A}:=\Lie(A^{\mathrm U})$ and $\caln^{\mathrm U,d}:=\Lie(R_u^{\mathrm U,d})$.
\begin{lem}\label{tru}
 The map $$\widetilde{\tr}:\mathfrak{n}^\mathrm U_{A}\times(\mathfrak{n}^\mathrm U_{{A}})^t\rightarrow\BF_q,\quad (A,B)\mapsto\tr(A\cdot B)^{q}+\tr(A\cdot B)$$ is a perfect pairing.
\end{lem}
\begin{proof}
    Let $X=\begin{pmatrix}
 0&B\\
 0&0 \end{pmatrix}\in\caln_{A}^{\mathrm{U}}$ and $Y=\begin{pmatrix}
 0&0\\
 A&0 \end{pmatrix}\in(\caln_{A}^{\mathrm{U}})^t$, where $B=\begin{pmatrix}
 B_1&B_2\\
 0&B_3 \end{pmatrix}$ and $A=\begin{pmatrix}
 A_1&0\\
 A_2&A_3 \end{pmatrix}$. By definition, $B_1=-J_d\overline{B_3^t}J_{n-d}$, $B_2=-J_d\overline{B_2^t}J_d$, $A_1=-J_{n-d}\overline{A_3^t}J_d$ and $A_2=-J_d\overline{A_2^t}J_d$. If $\widetilde{\tr}(X\cdot Y)=0$ for all $Y\in (\caln_{A}^{\mathrm{U}})^t$, we get $W=\mathrm{tr}(B_1A_1+B_2A_2+B_3A_3)+\mathrm{tr}(\overline{B_1A_1+B_2A_2+B_3A_3})=0$ for all $A$. First, letting $A_1=A_3=0$, we have $W=2\mathrm {tr}(B_2A_2)=0$. Thus,  $B_2=0$. Next, letting $A_2=0$, we have $W=2\mathrm{tr}(B_1A_1)+2\mathrm{tr}(\overline{B_1A_1})=0$. If $B_1\neq 0$, then there exists a matrix $C_1$ such that $\mathrm{tr}(B_1C_1)=c\neq 0$ and $c+c^q=0$. For all ${a}\in \BF_{q^2}$, we also have $ac+(ac)^q=0$. Thus, $a=a^q$ since $c\neq0$. That is a contradiction, and the map is non-degenerate.
\end{proof}
\section{Character Formulas via Coadjoint Orbits}\label{4}
   \subsection{Clifford theory and coadjoint orbits}\label{2.2}
   Let us review Clifford theory and coadjoint orbit theory \cite{Se77,Kir99}. 
	 Assume that a finite group $G =A\rtimes H$ is a semidirect product of $A$ by $H$, where $A$ is an abelian normal subgroup of $G$. Let $\widehat{A}$ be the set of all irreducible characters of $A$. Let $h\in H$ act on $\widehat A$ by $h\cdot \chi=\chi^h,$ where $\chi^h\in \widehat A$ is given by 
	 $$\chi^h(a)=\chi(h^{-1}ah) \text{ for }a\in A.$$
	 Let $H_\chi=\{h\in H\ |\ \chi^h=\chi\}$  be the stabilizer of $\chi$ in $H$. For $\chi\in\widehat A$, we identify it with a character of $A\rtimes H_{\chi}$ by $\chi(ah):=\chi(a)$, where $a\in A$ and $h\in H_{\chi}$. For $\tau\in \widehat H_\chi$, we identify it with a character of $A\rtimes H_{\chi}$ via the canonical projection $A\rtimes H_{\chi}\to H_{\chi}$.                       
	 Then the irreducible characters of $G$ can be described in terms of induced characters of the following form:
	 
	 \begin{thm}\cite[Section 8.2, Proposition 25]{Se77}\label{Clifford} For
	 	$\chi, \ \chi'\in\widehat A$, and $\tau, \ \tau'\in \widehat H_\chi,$ the following holds.
	 	\begin{enumerate}
	 		\item  
	 		$\Ind_{A\rtimes H_{\chi} }^G\chi \otimes\tau$ is an irreducible character of $G$. 
	 		\item $\Ind_{A\rtimes H_{\chi} }^G\chi \otimes\tau\cong \Ind_{A\rtimes H_{\chi'} }^G\chi' \otimes\tau'$ if and only if $\chi$ and $\chi'$ are in the same $H$-orbit, and $\tau=\tau'$.
	 		\item Every irreducible character $\pi$  of $G$ is isomorphic to 
	 		$\Ind_{A\rtimes H_{\chi}}^G\chi \otimes\tau,$ for some $\chi   \in\widehat A$, and $\tau  \in \widehat H_\chi.$
	 	\end{enumerate}	 	
	 \end{thm}
 	 Fix a nontrivial one-dimensional character $\psi$ of $\BF_q$. Note that $A^{X}=I+B_X$ is an abelian group with $B_X=\Lie(A^{X})= \caln_{A}^{X}$ ($X=\mathrm C ,\mathrm D$). By Lemma \ref{tr} and Lemma \ref{tr2},
	 $$\widehat A^{X}=\{\psi_T\  | \ T\in B_X^t\},$$ where $$\psi_T(I+x):=\psi(\tr Tx),\text{ for }x\in B_X.$$
     Also, we have $A^{\mathrm U}=I+B_{\mathrm U}$ is an abelian group with $B_{\mathrm U}=\Lie(A^{\mathrm U})= \caln_{A}^{\mathrm U}$. By Lemma \ref{tru},
     $$\widehat A^{\mathrm U}=\{\psi_T\  | \ T\in B_\mathrm U^t\},$$ where $$\psi_T(I+x):=\psi(\widetilde{\tr} Tx),\text{ for }x\in B_{\mathrm U}.$$
	 \begin{defn} 
	 For a matrix $c=(c_{i,j})\in \caln^{X,d}$,
	 denote by 
	 \begin{equation}\label{Aset}
	 \Supp(c):=\{(i, j) \mid c_{i,j}\ne 0\}.\end{equation}  
	 For a subset $M\subset \caln^{X,d}$, 
	 denote by 
	 \begin{equation}\label{set}
	 \Supp(M):=\bigcup_{c\in  M}\Supp(c).
	 \end{equation}  
	 	\end{defn} 
	For $A^{\mathrm C}=I+B_{\mathrm C}$, we may identify its dual space with $B_{\mathrm C}^t$ via the trace map.
	 The coadjoint action of $g\in H^\mathrm C$ on $\alpha\in B_{\mathrm C}^t$ in \cite{Kir99} is given by
	 \begin{equation}\label{eq:Kirillov}
	 (g\cdot \alpha)(I+b)=\tr(\alpha g^{-1}bg), \text{ for } b\in B_{\mathrm C}.	
	 \end{equation} 
	 It can be identified with $\alpha \mapsto [g\alpha g^{-1}]_{B_{\mathrm C}}$, where $[\cdot]_{B_{\mathrm C}}$ is  the projection from $\RMat_{2n}(\BF_q)$ to $B_{\mathrm C}^t$.
	 More precisely,
	 for $m=(m_{i,j})\in \bigcup_{g\in  H_{\mathrm C}}gB_{\mathrm C}^tg^{-1}$, 
	 \[
	 ([m]_{B_{\mathrm C}})_{i,j}=m_{i,j} \text{ if }(j,i)\in \Supp({B_{\mathrm C}}),
	 \text{ and } ([m]_{B_{\mathrm C}})_{i,j}=0	\text{ if }(j,i)\notin \Supp({B_{\mathrm C}}). 
	 \]
	 Define the action of $H^\mathrm C$ on $B_{\mathrm C}^t$ by
	 \begin{equation}     
	 g\circ_{B_{\mathrm C}} \alpha:=[g\alpha g^{-1}]_{B_{\mathrm C}}\text{ for }g\in H^\mathrm C\text{ and }\alpha\in B_{\mathrm C}^t.\end{equation}
	 Since $(g\cdot \alpha)(I+b)=\tr([g\alpha g^{-1}]_{B_{\mathrm C}}b)$ for all $b\in B_{\mathrm C}$,
	 the action $\circ_{B_{\mathrm C}}$ of $H^\mathrm C$ on $B_{\mathrm C}^t$ coincides with the {\itshape coadjoint action}, whose orbits are called {\itshape coadjoint orbits}. The definition of {\itshape coadjoint orbits} for type $\mathrm D$ are similar to that of type $\mathrm C$, while for type $\mathrm{U}$ it is necessary to replace $\widetilde{\tr}$ with $\tr$ in Equation \eqref{eq:Kirillov}.

\subsection{ Characters number formula} 
Recall ${R_u^{X,d}}\simeq A^{X}\rtimes H^{X}$.
When $X=\mathrm C$, we simply write $P_\mathrm C(B_1,B_3,B_2,0)\in\Lie(A^{\mathrm C})^t$  for $$\begin{pmatrix}
 0_{d\times d}&0&0&0\\
 0&0_{(n-d)\times (n-d)}&0&0\\
 B_1&B_3&0_{d\times d}&0\\
 B_2&0&0&0_{(n-d)\times (n-d)}
	 \end{pmatrix},$$ \\
where $B_1$ is symmetric and $B_3^t=B_2$.

 Let $$V_\mathrm C(A):=\begin{pmatrix}
 {I}_{d}&A&0&0\\
 0&{I}_{n-d}&0&0\\
 0&0&{I}_{d}&0\\
 0&0&{-A^t}&{I}_{n-d}
	 \end{pmatrix},$$\\
where $A$ be the matrix which take the values in $\BF_q$. Consider the coadjoint action of $V_{\mathrm C}(A)$ on $P_{\mathrm C}(B_1,B_3,B_2,0)$. That is, \begin{align*}
      &[V_{\mathrm C}(A)\cdot P_{\mathrm C}(B_1,B_3,B_2,0)\cdot V_{\mathrm C}(-A)]_{B_{\mathrm C}}\\&=P_{\mathrm C}(B_1,B_3,B_2,0)+[P_{\mathrm C}(0,-B_1A,-A^tB_1,A^tB_1A-B_2A-A^tB_2)]_{B_{\mathrm C}}.
  \end{align*}
  It is clear that
\begin{equation}\label{e01}
      [P_{\mathrm C}(0,-B_1A,-A^tB_1,A^tB_1A-B_2A-A^tB_2)]_{B_{\mathrm C}}=0 
  \end{equation}
  gives  \textit{algebraic} equations in variables $\mathbf{a}_{i,j}$ for $(i,j)\in\Supp(\caln_{H}^{\mathrm C})$, while\begin{equation}\label{e01}
      [P_{\mathrm C}(0,-B_1A,-A^tB_1,0)]_{B_{\mathrm C}}=0
  \end{equation}
  gives \textit{linear} equations.

   The results for ${R_u^{{\mathrm D},d}}$ are similar. The difference is $B_1 = -B_1^T$ and $B_2^T = -B_3$.

  When $X= \mathrm U$, $P_\mathrm C(B_1,B_3,B_2,0)$ is replaced by 
  $$P_ \mathrm U(B_1,B_3,B_2,0):=\begin{pmatrix}
 0_{d\times d}&0&0&0\\
 0&0_{(n-d)\times (n-d)}&0&0\\
 B_1&0&0_{(n-d)\times (n-d)}&0\\
 B_2&B_3&0&0_{d\times d}
	 \end{pmatrix},$$ \\
where $B_2J_d$ is skew-Hermitian and $B_1=-J_{n-d}\overline{B_3^t}J_d$. The matrix $V_\mathrm C(A)$ is replaced by 
$$V_ \mathrm U(A):=\begin{pmatrix}
 {I}_{d}&A_1&0&0\\
 0&{I}_{n-d}&0&0\\
 0&0&{I}_{n-d}&A_2\\
 0&0&0&{I}_{d}
	 \end{pmatrix},$$
where $A_1$ is be the matrix which take the values in $\BF_{q^2}$. Besides, $A_1=-J_d\overline{A_2^t}J_{n-d}$. Repeating the above process, we get \begin{equation}\label{e02}
      {[P_ \mathrm U(A_2B_2,-B_2A_1,0,0)]_{B_{\mathrm U}}=0},
  \end{equation}
 which gives \textit{linear} equations in variables $\mathbf{a}_{i,j}$ for $(i,j)\in\Supp(\caln_{H}^{\mathrm U})$. 
  We denote the coefficient matrix of (\ref{e01}) and (\ref{e02}) by $\ScP_X(B_1,B_3,B_2,0)$ under a fixed total order of $\mathbf{a}_{i,j}$ ($X=\mathrm C,\mathrm D, \mathrm U$).
  Define $$\mathfrak{P}_X({R_u}):=\{\ScP_X(B_1,B_3,B_2,0)\mid P_X(B_1,B_3,B_2,0)\in\Lie({A}^{X})^t\}.$$
  Let $$\pi_c: \Lie(A^{X})^t\to \mathfrak{P}_X({R_u}), \quad P_X(B_1,B_3,B_2,0)\mapsto \ScP_X(B_1,B_3,B_2,0)$$ be the canonical map.
  For $e\in\BN$, define 
  \begin{equation}\label{eqru}
     \mathfrak{P}_X({R_u},e):=\{\ScP_X(B_1,B_3,B_2,0)\in\frP_X({R_u})\mid \rank(\ScP_X(B_1,B_3,B_2,0)=e\}.
  \end{equation} Then $\mathfrak{P}_X({R_u})=\bigcup_{e=0}^{\infty}\mathfrak{P}_X({R_u},e)$ is a finite partition of $\mathfrak{P}_X({R_u})$. Denote the fiber of $\pi_c$ over $\mathfrak{P}_X({R_u}, e)$ by $\Tilde{\mathfrak{P}}_X({R_u}, e)$.

  Let $\psi: (\caln_{A}^{\mathrm X})^t\to \widehat A^{X},T\mapsto\psi_T$ ($X=\mathrm C, \mathrm D,\mathrm U$). We have the following lemma that establishes a one-to-one correspondence between the $H^{X}$-orbit on $\widehat A^{X}$ and the coadjoint orbits on $(\caln_{A}^{\mathrm X})^t$.
   \begin{lem}
       $\psi$ is compatible with the action of $H^{X}$ on $\widehat A^{X}$ and the coadjoint action of $H^{X}$ on $(\caln_{A}^{\mathrm X})^t$.
   \end{lem}
   \begin{proof}
       We give just the proof of type $\mathrm C$. Let $g\in H^{\mathrm C}$, $a\in (\caln_{A}^{\mathrm C})^t$. We have
$$\psi(g\circ_{B_{\mathrm C}} a)=\psi_{[gag^{-1}]_{B_{\mathrm C}}},$$
$$g\circ_{B_{\mathrm C}} \psi(a)=\psi_{[gag^{-1}]_{B_{\mathrm C}}}.$$
   \end{proof}

We no longer distinguish between the coadjoint orbit $\mathbb{O}_{a}$ in $(\caln_{A}^{\mathrm X})^t$ and its corresponding orbit $\mathbb{O}_{\psi _a}$ in $\widehat A^{X}$, as well as the stabilizer subgroups of $a\in (\caln_{A}^{\mathrm X})^t$ and $\psi _a\in \widehat A^{X}$.
    \begin{thm}[Characters number formula]\label{TYPE C}
        The number of irreducible characters of degree $|k|^e$ in ${R_u^{X,d}}$ is $$ |k|^{-2e}\cdot|\Tilde{\mathfrak{P}}_X({R_u}, e)|\cdot|H^X|.$$ Furthermore, the set of integers that occur as degrees of irreducible characters of ${R_u^{X,d}}$ is precisely $\{|k|^{e}\mid \mathfrak{P}_X({R_u},e)\neq \emptyset\}$ $($$|k|=q$ if $X=\mathrm C$ or $\mathrm D$, $|k|=q^2$ if $X=\mathrm{U}$$)$.
    \end{thm}
    \begin{proof}
         We give just the proof of ${R_u^{\mathrm C,d}}$ ($k=\BF_q$). Fix a nontrivial one-dimensional character $\psi$ of $\BF_{q}$. Then $\hat{A}^\mathrm C=\{\psi_{T}\mid T\in(\caln_{A}^{\mathrm C})^t\}$,  which we identify $\hat{A}^\mathrm C$ with $(\caln_{A}^{\mathrm C})^t$ via $T\leftrightarrow \psi_{T}$ by Lemma \ref {tr}. Moreover, the action of $H^\mathrm C$ on $\hat{A}^\mathrm C$ coincides with the coadjoint action of $H^\mathrm C$ on $(\caln_{A}^{\mathrm C})^t$. Thus, we can identify the $H^\mathrm C$-orbits on $\hat{A}^\mathrm C$ with the coadjoint orbits of $H^\mathrm C$ on $(\caln_{A}^{\mathrm C})^t$. For a coadjoint orbit $\mathbb{O}_{a}$ with $a\in\mathbb{O}_{a}$, let $H_a^\mathrm C$ be the stabilizer of $a$ in $H^\mathrm C$. Clearly, $H_a^\mathrm C$ is an abelian group, and $\hat{H}_a^\mathrm C$ consists of $|{H}_a^\mathrm C|$ one-dimensional irreducible characters. By Theorem \ref{Clifford}, $\BO_{a}$ gives $|{H}_a^\mathrm C|$ irreducible characters of ${R_u^{\mathrm C,d}}$ of degree $[H^\mathrm C:{H}_a^\mathrm C]$, where $[H^\mathrm C:{H}_a^\mathrm C]$ is the index of ${H}_a^\mathrm C$ in $H^\mathrm C$. We write $$V_\mathrm C(A)=I+\begin{pmatrix}
             0_{d\times d}&A&0&0\\
             0&0_{(n-d)\times (n-d)}&0&0\\
             0&0&0_{d\times d}&0\\
             0&0&-A^t&0_{(n-d)\times(n-d)}
         \end{pmatrix}\in H^\mathrm C,$$ where $A=(a_{i,j})$ , $a_{i,j}=0$ if $(i,j)\notin\Supp(\caln_{H}^{\mathrm U})$. We have $${H}_a^\mathrm C=\{V_\mathrm C(A)\in H^\mathrm C\mid (a_{i,j})\ \text{is a solution of } (\ref{e01})\}.$$ We write $a=P_\mathrm C(B_1,B_3,B_2,0)$, then $|{H}_a^\mathrm C|=q^{-\rank(\ScP_\mathrm C(B_1,B_3,B_2,0))}\cdot |H^\mathrm C|$, and the degree of irreducible characters associated with $\BO_{a}$ is $q^{\rank(\ScP_C(B_1,B_3,B_2,0)}$. We denote the number of $H^\mathrm C$-coadjoint orbits with stabilizer (of a representative element) of index $q^{e}$ by $N_{e}$, then the number of irreducible characters of ${R_u^{\mathrm C,d}}$ of degree $q^{e}$ is $$N_{e}\cdot q^{-e}\cdot |H^\mathrm C|.$$
        Let $$\calO_{e}:=\bigcup\limits_{[H^\mathrm C:{H}_a^\mathrm C]=q^{e}}\BO_{a}$$ be the set of union of all coadjoint orbits whose stabilizer in $H^\mathrm C$ has index $q^{e}$. Then $$\calO_{e}=\{P_\mathrm C(B_1,B_3,B_2,0)\in(\caln_{A}^{\mathrm C})^t\mid \rank(\ScP_\mathrm C(B_1,B_3,B_2,0))=e\},$$ which is just the set $\tilde{\frP}_\mathrm C({R_u},e)$. Since the cardinality of an $H^\mathrm C$-orbit is equal to the index of its stabilizer, the coadjoint orbits in $\calO_{e}$ have the same cardinality $q^{e}$. Hence $N_{e}=q^{-e}\cdot |\Tilde{\frP}_\mathrm C({R_u},e)|$, and the number of irreducible characters of ${R_u^{\mathrm C,d}}$ of degree $q^e$ is 
        $$q^{-2e}\cdot|\Tilde{\frP}_\mathrm C({R_u},e)|\cdot|H^\mathrm C|.$$ 
        
        The rest part is clear, since the map $\pi_c$ is surjective and all irreducible characters can be given as above. The proof for ${R_u^{\mathrm D,d}}$ follows a pattern similar to the one presented here. For ${R_u^{\mathrm U,d}}$, it is necessary to use Lemma \ref{tru} instead of Lemma \ref{tr}. Furthermore, we must replace $q$ with $q^2$ in the relevant equations. Following these adjustments, the resultant formula remains consistent with the original.
    \end{proof} 

  
\begin{rmk}
   In the case of $d=n$, the group ${R_u^{X,n}}$ is abelian. Similarly, in the case of $d=n-1$, the group ${R_u^{\mathrm D,n-1}}$ is also abelian. It is clear that the number of irreducible characters is equal to the number of their elements.
\end{rmk}
\section{The polynomial property of ${R_u^{X,d}}$}\label{geo poly}

\subsection{Polynomial property for unipotent radicals of maximal parabolic subgroups}
We define ${R_u^{X,d}}$ as having the \textit{polynomial property} if the number of its irreducible characters with degree $q^e$ can be expressed as a polynomial in $q$ with integral coefficients for any given $e\in\BN$.


According to Theorem \ref{TYPE C}, the problem of determining the polynomial property of ${R_u^{X,d}}$ can be simplified to assess the polynomial property of $\tilde{\frP}_X({R_u},e)$ instead. In other words, ${R_u^{X,d}}$ possesses a polynomial property if and only if $|\tilde{\frP}_X({R_u},e)|$ is a polynomial in $q$ with integral coefficients.




\begin{lem}\label{sp2n}\cite[Theorem 2]{Mac69}
   Let $N(n,r)$ denote the number of symmetric matrices of size $n\times n$ and rank $r$, with entries in $\BF_q$. Then,
\begin{equation}\label{N2e}
     N(n,2s)=\prod \limits_{i=1}^s \frac{q^{2i}}{q^{2i}-1} \prod \limits_{i=0}^{2s-1} (q^{n-i}-1),
\end{equation}
 \begin{equation}\label{N2o}
    N(n,2s+1)=\prod \limits_{i=1}^s \frac{q^{2i}}{q^{2i}-1} \prod \limits_{i=0}^{2s} (q^{n-i}-1).
\end{equation}
\end{lem}

\begin{lem}\label{so2n}
   Let $S(n,2s)$ denote the number of skew-symmetric matrices in $\mathrm{Mat}_{n}(\BF_q)$ with rank $2s$. We have
   \begin{equation}\label{s2e}
   S(n,2s)=\frac{q^{s^2-s}\prod \limits_{i=0}^{2s-1} (q^{{n}-i}-1)} {\prod \limits_{i=1}^{s} (q^{2i}-1)}. 
\end{equation}
\end{lem}
\begin{proof}
    Let $S:= S(n,2s)$. Consider $G=GL(n,q)$ action on $S$ via congruence. That is to say, $P$ actions on $X$ is $PXP^t$ ($P\in G$, $X\in S$) and this action is transitive because any two skew-symmetric matrix with the same rank are congruent. Let $G_X=\{P\in G\mid P\cdot X=X\}$ . Without loss of generality, let $$X=\begin{pmatrix} 
 0&\mathrm{Id}_{s}&0\\
 -\mathrm{Id}_{s}&0&0\\
 0&0&0_{n-2s}\\
 \end{pmatrix}=\begin{pmatrix} 
 E_{s}&0\\
 0&0\\
 \end{pmatrix},$$
 where $E_s=\begin{pmatrix} 
 0&\mathrm{Id}_{s}\\
 -\mathrm{Id}_{s}&0\\
 \end{pmatrix}$ is the matrix representing the standard non-degenerate symplectic form on a $2s$-dimensional space. Let $P\in GL(n,q)$ and decompose it as a block matrix $$\begin{pmatrix} 
 {A}_{2s\times 2s}&B_{2s\times (n-2s)}\\
 C_{ (n-2s)\times 2s}&D_{(n-2s)\times (n-2s)}\\
 \end{pmatrix}.$$
  $PXP^t=X$ if and only if $$\begin{pmatrix} 
 AE_sA^t&AE_sC^t\\
CE_sA^t&CE_sC^t\\
 \end{pmatrix}=\begin{pmatrix} 
 E_s&0\\
0&0\\
 \end{pmatrix}.$$
The equation $AE_sA^t=E_s$ implies that $A$ is symplectic and in particular is of full rank. Then, the matrix $AE_sC^t=0$ implies that $C=0$, since $AE_s$ has full rank. Hence, $$G_X=\left\{\begin{pmatrix} 
 A_{2s\times 2s}&B_{2s\times (n-2s)}\\
 0&D_{(n-2s)\times (n-2s)}\\
 \end{pmatrix}\middle|\begin{array}{c}A\in \mathrm {Sp}_{2s}(q),B\in\Mat_{2s\times (n-2s)}(q),D\in GL(n-2s,q)
    \end{array}\right\}.$$
So we have 
\begin{align}\label{s2e1}
|S| &= \frac{|GL(n,q)|}{|\mathrm {Sp}_{2s}(q)|\cdot|\Mat_{2s\times (n-2s)}(q)|\cdot|GL(n-2s,q)|} \nonumber \\
&= \frac{q^{s^2-s}\prod \limits_{i=0}^{2s-1} (q^{n-i}-1)} {\prod \limits_{i=1}^{s} (q^{2i}-1)}.\nonumber
\end{align}
    
\end{proof}
\begin{lem}\label{Con}
    For any skew-Hermitian matrix $C$ of size $n\times n$ and rank $r$, with entries in $\BF_{q^2}$. There exists an invertible matrix $A$ such that
    \begin{equation}
        AC\overline{A^t}=\alpha\begin{pmatrix} 
 I_r&0\\
0&0\\
 \end{pmatrix}
    \end{equation}
where $\alpha\neq0$ satisfies $\alpha+\alpha^q=0$.
\end{lem}

\begin{lem}\label{U2n}
   Let $U(n,r)$ denote the number of skew-Hermitian matrices of size $n\times n$ and rank $r$, with entries in $\BF_{q^2}$. Then,  
   \begin{equation}\label{Ur}
   U(n,r)=(q-1)q^{\frac{r(r-1)}{2}}\frac{\prod \limits_{i=n-r+1}^{n} (q^{2i}-1)} {\prod \limits_{s=1}^{r} (q^s-(-1)^s)}. 
\end{equation}
\end{lem}
\begin{proof}
    Let $U_\alpha (\alpha\neq0)$ be the subset of $U(n,r)$ consisting of all elements of $U(n,r)$ that are congruent to $X=\alpha\begin{pmatrix} 
 I_r&0\\
0&0\\
 \end{pmatrix}$. Consider the action of $G=GL(n,{q^2})$ on $U_\alpha$ via  congruence which is transitive by Lemma \ref{Con}. Let $G_X=\{P\in G\mid P\cdot X=X\}$. Picking a matrix $P\in GL(n,{q^2})$ and decompose it as a block matrix $$\begin{pmatrix} 
 {A}_{r\times r}&B_{r\times (n-r)}\\
 C_{ (n-r)\times r}&D_{(n-r)\times (n-r)}\\
 \end{pmatrix}.$$
  $PX\overline{P^t}=X$ if and only if $$\alpha\begin{pmatrix} 
 A\overline{A^t}&A\overline{C^t}\\
C\overline{A^t}&C\overline{C^t}\\
 \end{pmatrix}=\alpha\begin{pmatrix} 
 I_r&0\\
0&0\\
 \end{pmatrix}.$$
The equation $A\overline{A^t}=I_r$ implies that $A$ is unitary and in particular is of full rank. Then, the matrix $A\overline{C^t}=0$ implies that $C=0$, since $A$ has full rank. Hence $$G_X=\left\{\begin{pmatrix} 
 A_{r\times r}&B_{r\times (n-r)}\\
 0&D_{(n-r)\times (n-r)}\\
 \end{pmatrix}\middle| \begin{array}{c}A\in \mathrm{U}_r(q^2),B\in\Mat_{r\times (n-r)}({q^2}),D\in GL(n-r,{q^2})\\ 
    \end{array}\right\}.$$\\
So we have  \\\begin{align}\label{s2e2}
|U(n,r)| &= (q-1)\frac{|GL(n,{q^2})|}{|\mathrm{U}_r(q^2)|\cdot|\Mat_{r\times (n-r)}(q^2)|\cdot|GL(n-r,{q^2})|} \nonumber \\
&= (q-1)q^{\frac{r(r-1)}{2}}\frac{\prod \limits_{i=n-r+1}^{n} (q^{2i}-1)} {\prod \limits_{s=1}^{r} (q^s-(-1)^s)} .\nonumber
\end{align}
    
\end{proof}

Denote the number of irreducible characters of ${R_u^{X,d}}$  with degree $q^e$ by $N_{(X,d),e}$.
\begin{thm}\label{character number of unipotent radical of maximal parabolic gruoup C}
For any $1\le d< n$, $0\le r\le d$. Let $e=(n-d)r$, we have \begin{equation}
N_{(\mathrm C,d),e}=\left\{
\begin{array}{rcl}
q^{2d(n-d)-2e}\cdot\prod \limits_{i=1}^s \frac{q^{2i}}{q^{2i}-1} \prod \limits_{i=0}^{2s-1} (q^{d-i}-1) & & {if\ e=2s(n-d)},\\
q^{2d(n-d)-2e}\cdot\prod \limits_{i=1}^s \frac{q^{2i}}{q^{2i}-1} \prod \limits_{i=0}^{2s} (q^{d-i}-1) & & {if\ e=(2s+1)(n-d)}.\\
\end{array} \right.
\end{equation}
Moreover, the set of integers that occur as degrees of irreducible characters of  $R_u^{\mathrm C,d}$ is precisely $\{q^e\mid 0\le r\le d,\ e=(n-d)r\}$. 
\end{thm}

\begin{proof}	 	
The coefficient matrix of (\ref{e01}) is $\ScP_\mathrm C(B_1,B_3,B_2,0)$  under a fixed total order of $\mathbf{a}_{i,j}$, where 
$$\ScP_\mathrm C(B_1,B_3,B_2,0)=\begin{pmatrix}
B_1 & &\\
&	\ddots   &    \\
&	  &    B_1   \\
	 \end{pmatrix}$$ is a block diagonal matrix of degree $d(n-d)$. Thus, $|\tilde{\frP}_\mathrm C({R_u},e)|=N(d,r)q^{d(n-d)}$. 
By Proposition \ref{TYPE C} and Lemma \ref{sp2n}, we complete the proof.
\end{proof}
 The proof of the following theorem is similar to that of the above theorem.
\begin{thm}\label{character number of unipotent radical of maximal parabolic gruoup D}
For any $1\le d< {n-1}$, $0\le2s\le d$. Let $e=2(n-d)s$, we have \begin{equation}
   N_{(\mathrm D,d),e}=q^{2d(n-d)-2e+s^2-s}\cdot\frac{\prod \limits_{i=0}^{2s-1} (q^{d-i}-1)} {\prod \limits_{i=1}^{s} (q^{2i}-1)}.
\end{equation}
Moreover, the set of integers that occur as degrees of irreducible characters of $R_u^{\mathrm D,d}$ is precisely $\{q^e\mid 0\le 2s\le d,\ e=2(n-d)s\}$. 
\end{thm}
\begin{proof}	 	
By Proposition \ref{TYPE C} and Lemma \ref{so2n}.
\end{proof}
\begin{rmk}
  The method is not applicable for $\mathrm {SO}_{2n+1}(q)$ due to our inability to factor $R_u^{\mathrm B,d}$ into a semi-direct product of two abelian groups.
\end{rmk}

\begin{thm}\label{character number of unipotent radical of maximal parabolic group U}
For any $0\le d\le {n-1}$, $0\le r\le {d}$. Let $e=(n-d)r$, we have \begin{equation}
    N_{(\mathrm U,d),e}=q^{4d(n-d)-4e}(q-1)q^{\frac{r(r-1)}{2}}\cdot\frac{\prod \limits_{i=d-r+1}^{d} (q^{2i}-1)} {\prod \limits_{s=1}^{r} (q^s-(-1)^s)} .
\end{equation}
Moreover, the set of integers that occur as degrees of irreducible characters of  $R_u^{\mathrm U,d}$ is precisely $\{q^e\mid 0\le r\le {d},\ e=(n-d)r\}$. 
\end{thm}
\begin{proof}	 
The coefficient matrix of (\ref{e02}) is 
$$\ScP_\mathrm U(B_1,B_3,B_2,0)=\begin{pmatrix}
B_2 & &\\
&	\ddots   &    \\
&	  &    B_2   \\
	 \end{pmatrix},$$
    which is a block diagonal matrix of degree $d(n-d)$.
Thus, $|\tilde{\frP}_\mathrm U({R_u},e)|=U(d,r)q^{2d(n-d)}$. 
By Proposition \ref{TYPE C} and Lemma \ref{U2n}, we complete the proof.
\end{proof}

\section{GEOMETRIC CORRESPONDENCE}\label{GC}
In this section, we establish a geometric correspondence for the matrix mentioned above to derive a stronger result.

We provide some preliminary content about Schubert varieties. For more details, see \cite{Bor12, Hum12, Lak08 }.
\subsection{Schubert varieties in $G/Q$}
 Let $G$ be a semisimple and simply connected algebraic group defined over $k$. We use the notation in Section \ref{section 2.2.1}. Denote by $W_Q$ the Weyl group of $Q$. In each coset $wW_Q\in W/W_Q$, there exists a unique element of minimal length. Let $W^Q$ be the set of minimal representatives of $W/W_Q$. For $w\in W$, we denote the coset $wQ$ in $G/Q$ by $e_{w,Q}$. We have Bruhat decomposition $$
G/Q = \bigsqcup_{w\in W^Q} Be_{w,Q}.
$$
 Each $Be_{w,Q}$ is called a Bruhat cell, which is isomorphic to an affine space. Let $X_Q(w)$ be the Zariski closure of $Be_{w,Q}$ in $G/Q$. Then $X_Q(w)$ with the canonical reduced structure is called the Schubert variety. It is well-known that the Schubert variety $X_Q(w)$ can be written as $X_Q(w) = \bigsqcup_{u\in W^Q,u\leq w} Be_{u,Q}$, where $\leq$ is the Bruhat order. Let $B^-$ be the Borel subgroup of $G$ opposite to $B$. We have opposite decomposition
\[ G/Q = \bigsqcup_{w \in W^Q} B^- e_{w,Q}. \]
Here $B^- e_{w,Q}$ is called an opposite Bruhat cell, which is also an affine space. Consider the canonical projection $\pi: G/B \to G/Q$, we have the following properties: \begin{itemize}
    \item [(1)] $\pi^{-1}(B^- e_{w,Q}) = \bigsqcup_{u \in wW_Q} B^- e_{u,B}$.
    \item [(2)]  For $w \in W^Q$, the restriction of $\pi$ to $Be_{w,B}$ is an isomorphism onto $Be_{w,Q}$.
\end{itemize}

For $u, v \in W$, denote $X_u^v := Be_{u,B} \cap B^- e_{v,B}$ to be the intersection of $Be_{u,B}$ and $B^- e_{v,B}$; for $u, v \in W^Q$, define $Y_u^v := Be_{u,Q} \cap B^- e_{v,Q}$ to be the intersection of $Be_{u,Q}$ and $B^- e_{v,Q}$. $X_u^v$ and $Y_u^v$ are nonempty if and only if $v \leq u$. All varieties above can be defined over arbitrary field. Fix a base field $\mathbb{F}_q$. Consider the varieties $Y_u^v$ and $X_u^v$ that is defined over $\mathbb{F}_q$. We have :
 \begin{lem}\label{BQ}
 $|Y_u^v(\mathbb{F}_q)| = \sum_{w \in vW_Q, w \leq u} |X_u^w(\mathbb{F}_q)|$.
 \begin{proof}
     Noting that the isomorphism $\pi|_{Be_{u,B}}: Be_{u,B} \to Be_{u,Q}$ is defined over $\mathbb{F}_q$, we have
\[ Y_u^v \simeq (\pi|_{Be_{u,B}})^{-1}(Y_u^v) = Be_{u,B} \cap (\pi^{-1}(B^-e_{v,Q})) = \bigsqcup_{w \in vW_Q, w \leq u} X_u^w. \]
Hence
\[ |Y_u^v(\mathbb{F}_q)| = \sum_{w \in vW_Q, w \leq u} |X_u^w(\mathbb{F}_q)|. \]
 \end{proof}
 \end{lem}
The $B$-orbit $Be_{w_0, Q}$ in $G/Q$ ($w_0$ being the unique element of maximal length in $W$) is called the big cell in $G/Q$. It is a dense open subset of $G/Q$, and it gets identified with $R_u(Q)$, the unipotent radical of $Q$. Let $B^-$ be the Borel subgroup of $G$ opposite to $B$. The $B^-$-orbit $B^-e_{id, Q}$ is called the opposite big cell in $G/Q$. This is again a dense open subset of $G/Q$, and it gets identified with the unipotent subgroup of $B^-$ associated with $S_Q$. Observe that both the big cell and the opposite big cell can be identified with affine space.

For a Schubert variety $X_Q(w) \subset G/Q$, $B^-e_{{id},Q} \cap X_Q(w)$ is called the opposite cell in $X_Q(w)$ (by abuse of language). In general, it is not a cell. It is a non-empty open affine subvariety of $X_Q(w)$ and a closed subvariety of the affine space $B^- e_{\text{id}}$. Note that the big cell $Be_{w_0,Q}$ being a $B$-orbit has empty intersection with any proper Schubert variety of $G/Q$. We
denote $B^- e_{\text{id},Q} \cap X_Q(w)$ by $Y_Q(w)$.

\subsubsection{The $\mathrm{Sp}_{2n}(q)$ case} 
  Let $G=\mathrm{Sp}_{2n}(q)$ and $Q=P^{\mathrm C,n}$. Denote by $R^-_u(Q)$ the unipotent subgroup of $B^-$ associated with $S_Q$. We already know that $R^-_u(Q)$ can be identified with the opposite big cell $B^-e_{id, Q}$ in $G/Q$. In this case, $$R^-_u(Q)=\left\{\begin{pmatrix}
 I_n&0\\ 
 A&I_n \end{pmatrix}\mid A=A^t \right\}.$$ So, we can identify the symmetric matrices with $B^-e_{id, Q}$. Furthermore, the rank conditions on the matrices are given by intersecting with certain Schubert varieties $X_Q(w)$ (\cite[ Theorem 6.2.5.2]{Lak08}). More precisely, let $Y_Q(w)$ ($Y_Q(u)$, respectively) be identified with all symmetric matrices of rank at most $r$ ($r-1$, respectively) for a certain $w$($u\leq w$, respectively). Then, $N(n,r)$ is equal to the number of $\BF_q$-rational points of the intersection of $B^-e_{id, Q}$ with $\bigsqcup_{u\textless y\leq w  ,y\in W^Q} Be_{y,Q}$. Deodhar defined a polynomial $R_{w,v}(x)$ in \cite{Deo85} which generalizes the $R$-polynomials of Kazhdan-Lusztig. They have the property that $R_{w,v}(q)=|B^-e_{v, B}\cap Be_{w,B}|$. Furthermore, $R_{w,v}(q)$ is actually a polynomial in $q-1$ with non-negative integral coefficients \cite[Theorem 1.3]{Deo85}. Thus, $N(n,r)$ are polynomials in $q-1$ with non-negative integral coefficients by Lemma \ref{BQ}.
\subsubsection{The $\mathrm{SO}_{2n}(q)$ case}
Let $G=\mathrm{SO}_{2n}(q)$ and $Q=P^{\mathrm D,n}$. We can identify the skew-symmetric matrices with $B^-e_{id, Q}$. This is analogous to the preceding exposition. The rank conditions on the matrices are given by intersecting with certain Schubert varieties $X_Q(w)$(\cite[ Theorem 7.2.5.2]{Lak08}). Thus, $S(n,2s)$ are also polynomials in $q-1$ with non-negative integral coefficients.
\subsubsection{The $\mathrm{U}_{2n}(q^2)$ case}
We use the notation in Section \ref{section 2.2.2}. Fix the integers $0\le r\le {n}$ and let $V=\mathbb F_{q^2}^n$. The Grassmannian $G_{r,n}$ is the set of $r$-dimensional subspaces $ U\subseteq V$.

Let $G={GL}(n,q^2)$, $S_Q=S\backslash\{\alpha_r\}$. We have $G/Q \cong G_{r,n}$; see \cite{Lak08}. We want to prove that $U(n,r)$ are polynomials in $q-1$ with non-negative integral coefficients. It is enough to prove it for $$\frac{\prod \limits_{i=n-r+1}^{n} (q^{2i}-1)} {\prod \limits_{s=1}^{r} (q^s-(-1)^s)}=\prod \limits_{s=1}^{r} (q^{s}+(-1)^s)\frac{\prod \limits_{i=n-r+1}^{n} (q^{2i}-1)} {\prod \limits_{s=1}^{r} (q^{2s}-1)}.$$ The latter factor is the number of points in $G_{r,n}$. Thus, the problem becomes the question of whether $|G/Q|$ is a polynomial in $q-1$ with non-negative integral coefficients. It is a well-known result. 


 In summary, we have the following theorem:
\begin{thm}
    The functions $N_{(X,d),e}$ are polynomials in $q-1$ with non-negative integral coefficients.
\end{thm}

\end{document}